\documentclass[11pt]{article}
\usepackage{epsfig, latexsym, amsthm}
\usepackage{times,amssymb,amsmath}
\usepackage[usenames,dvipsnames]{xcolor}
\usepackage[text={16cm,23cm},centering]{geometry}
\usepackage{tikz}
\usetikzlibrary{matrix, calc, arrows}
\usepackage{xifthen}
\usepackage{authblk}

\title{Minimal classes of graphs of unbounded clique-width\\ defined by finitely many forbidden induced subgraphs}

\author[*]{A. Atminas}
\affil[*]{\small Department of Mathematics\\ London School of Economics, London, UK}
\author[$\dagger$]{R. Brignall$^{1,}$}
\affil[$\dagger$]{\small School of Mathematics and Statistics\\ The Open University, Milton Keynes, UK}
\author[$\ddagger$]{V. Lozin}
\affil[$\ddagger$]{\small Mathematics Institute\\University of Warwick, Coventry, UK}
\author[ ]{J. Stacho}

\footnotetext[1]{Partially supported by EPSRC Grant EP/J006130/1.}
\date{}

\newtheorem{theorem}{Theorem}[section]
\newtheorem{lemma}[theorem]{Lemma}
\newtheorem{claim}[theorem]{Claim}

\newtheorem{conj}[theorem]{Conjecture}
\newtheorem{prop}[theorem]{Proposition}

\tikzstyle{every node}=[circle, draw, fill=black,
                        inner sep=0pt, minimum width=4pt]
\newcommand{\mnnodearray}[2]{ 
\foreach \i in {1,...,#1}
	\foreach \j in {1,...,#2}
		\node at (\i,\j) {};
}
\newcommand{\uphoriz}[2]{ 
\pgfmathtruncatemacro\rightcol{#1 + 1}
\foreach \i in {1,...,#2}
	\foreach \j in {\i,...,#2}
		\draw (#1,\i) -- (\rightcol,\j);
}
\newcommand{\downhoriz}[2]{ 
\pgfmathtruncatemacro\rightcol{#1 + 1}
\foreach \i in {1,...,#2}
	\foreach \j in {\i,...,#2}
		\draw (#1,\j) -- (\rightcol,\i);
}
\newcommand{\upnohoriz}[2]{ 
\pgfmathtruncatemacro\rightcol{#1 + 1}
\pgfmathtruncatemacro\rowminus{#2 - 1}
\foreach \i in {1,...,\rowminus}
	{
	\pgfmathtruncatemacro\ip{\i+1}
	\foreach \j in {\ip,...,#2}
		\draw (#1,\i) -- (\rightcol,\j);
	}
}
\newcommand{\downnohoriz}[2]{ 
\pgfmathtruncatemacro\rightcol{#1 + 1}
\pgfmathtruncatemacro\rowminus{#2 - 1}
\foreach \i in {1,...,\rowminus}
	{
	\pgfmathtruncatemacro\ip{\i+1}
	\foreach \j in {\ip,...,#2}
		\draw (#1,\j) -- (\rightcol,\i);
	}
}
\newcommand{\curvetop}[3]{ 
	\draw[ultra thick] (#1,#3) edge[out=90,in=90] (#2,#3);
	\draw[thick,color=white] (#1,#3) edge[out=90,in=90,dashed] (#2,#3);
}
\newcommand{\labelrow}[2]{ 
\foreach \i in {1,...,#2}
	\node at (\i,#1) [label=below:\i] {};
}
\newcommand{\Free}{\operatorname{Free}}
\newcommand{\col}{\operatorname{col}}
\newcommand{\row}{\operatorname{row}}
\newcommand{\cwd}{\operatorname{cwd}}
\newcommand{\rwd}{\operatorname{rwd}}

\newcommand{\C}{\mathcal{C}}

\begin{document}
\maketitle
\begin{abstract}
We discover new hereditary classes of graphs that are minimal (with respect to set inclusion) of unbounded clique-width. The new examples include split permutation graphs and bichain graphs. Each of these classes is  characterised by a finite list of minimal forbidden induced subgraphs. These, therefore, disprove a conjecture due to Daligault, Rao and Thomass\'e from 2010 claiming that all such minimal classes must be defined by infinitely many forbidden induced subgraphs.

In the same paper, Daligault, Rao and Thomass\'e make another conjecture that every hereditary class of unbounded clique-width must contain a labelled infinite antichain. We show that the two example classes we consider here satisfy this conjecture. Indeed, they each contain a canonical labelled infinite antichain, which leads us to propose a stronger conjecture: that every hereditary class of graphs that is minimal of unbounded clique-width contains a \emph{canonical} labelled infinite antichain.
\end{abstract}


\section{Introduction}
In this paper, we study two notions: clique-width and well-quasi-ordering.
The first of them is a representative of the rich world of graph width parameters,
which includes both parameters studied in the literature for decades, such as path-width \cite{robertson:graph-minors-i:} or 
tree-width \cite{robertson:graph-minors-iii:}, and those that have been introduced recently, such as 
Boolean-width \cite{bui-xuan:boolean-width:} or plane-width \cite{kaminski:the-plane-width:}. 
The notion of clique-width belongs to the middle generation of graph width parameters.
The importance of this and many other parameters is due to the fact that 
many difficult algorithmic problems become tractable when restricted to graphs 
where one of these parameters is bounded by a constant.

The second notion of our interest is well-quasi-ordering. This is a highly desirable property and frequently discovered
concept in mathematics and theoretical computer science \cite{kruskal:the-theory-of-w:}. However, only a few examples of quasi-ordered sets possessing this property are available in the literature. One of the most remarkable results in this area is the proof of Wagner's conjecture stating that the set of all finite graphs is well-quasi-ordered by the minor relation
\cite{robertson:graph-minors-xx:}. However, the induced subgraph relation is not a well-quasi-order, 
because it contains infinite antichains, for instance, the set of all cycles. On the other hand, for graphs in particular classes this relation may become a well-quasi-order, which is the case, for instance, for cographs \cite{damaschke:induced-subgrap:}, and $k$-letter graphs \cite{petkovsek:letter-graphs-a:}. 
\pagebreak

In fact, both cographs and $k$-letter graphs possess a stronger property, namely that they are well-quasi-ordered with respect to the \emph{labelled} induced subgraph relation. If $(W,\leq)$ denotes a quasi-ordered set of labels, a \emph{labelling} of a graph $G$ is a function $\ell:V(G)\rightarrow W$, and we call the pair $(G,\ell)$ a \emph{labelled graph}. A labelled graph $(H,k)$ is a \emph{labelled induced subgraph} of $(G,\ell)$ if $H$ is isomorphic to an induced subgraph of $G$, and the isomorphism maps each vertex $v\in H$ to a vertex $w\in G$ so that $k(v)\leq \ell(w)$ in $W$. We say that a class of (unlabelled) graphs is \emph{labelled well-quasi-ordered} if it contains no set of graphs that can form an infinite antichain when the vertices are labelled from a set $(W,\leq)$ that is itself well-quasi-ordered.\footnote{By a slight abuse of notation, we will in future say that an unlabelled class of graphs \emph{contains} a labelled infinite antichain to mean that the class contains an infinite set of graphs that can be labelled in such a way as to form a labelled infinite antichain.}

For classes of graphs that are not labelled well-quasi-ordered, a natural question to ask is whether there exist subclasses that \emph{are}. In this line of enquiry, Guoli Ding~\cite{ding:on-canonical-antichains:} was the first to observe a phenomenon whereby a particular antichain is `unique' in its class, in the sense that well-quasi-ordering for subclasses can be determined by the presence or absence of elements of this antichain. Formally speaking (and extending the notion to labelled well-quasi-ordering), we say that a labelled infinite antichain $\cal A$ in a hereditary class $\C$ is \emph{canonical} if any hereditary subclass of $\C$ that has only finite intersection with (the underlying unlabelled graphs of) $\cal A$ is labelled well-quasi-ordered. 

In this paper, we will prove the following result. For the formal definitions of the two classes and of clique-width, see Section~\ref{sec:pre}.

\begin{theorem}\label{thm:main}The classes of bichain graphs and split permutation graphs are minimal hereditary classes (with respect to inclusion) of unbounded clique-width, and each contains a canonical labelled  infinite antichain that uses two incomparable labels.\end{theorem}

As both the class of bichain graphs and the class of split permutation graphs are characterised by having only finitely many minimal forbidden induced subgraphs (see Section~\ref{sec:pre}), Theorem~\ref{thm:main} disproves a conjecture due to Daligault, Rao and Thomass\'e~\cite[Conjecture 8]{daligault:well-quasi-orde:}, which stated that every minimal class (with respect to set inclusion) of unbounded clique-width must be characterised by \emph{infinitely} many minimal forbidden induced subgraphs.

In the same paper, Daligault, Rao and Thomass\'e proposed another conjecture that relates labelled well-quasi-ordering to boundedness of clique-width. We will postpone further discussion about this until the concluding remarks, except to say that it is this conjecture that motivates us to study clique-width and (labelled) well-quasi-ordering simultaneously.

The rest of this paper is organised as follows. In Section~\ref{sec:pre} we introduce basic definitions, including results about the two classes we are considering that allows us to focus solely on bichain graphs and ignore split permutation graphs thereafter. In Section~\ref{sec:universal} we will review existing constructions of so-called `universal' graphs for various classes, as well as introducing one further construction for the class of bichain graphs. In Section~\ref{sec:main-j} we will establish that the bichain graphs are minimal with unbounded clique-width, and in Section~\ref{sec:canonical} we will show that the bichain graphs contain a canonical 2-labelled infinite antichain. Section~\ref{sec:con} contains some concluding remarks, including our proposed strengthening of Daligault, Rao and Thomass\'e's conjecture that relates clique-width to well-quasi-ordering.

\section{Preliminaries} 
\label{sec:pre}

All graphs in this paper are simple, i.e.\ undirected, without loops and multiple edges.
We denote by $V(G)$ and $E(G)$ the vertex set and the edge set of a graph $G$, respectively.
Given a vertex $v\in V(G)$, we denote by $N_G(v)$, or simply $N(v)$ where the context is clear, the \emph{neighbourhood} of $v$, i.e.\ the set of vertices adjacent to $v$. 

Given a graph $G$ and a subset $U\subseteq V(G)$, we denote by $G[U]$ the subgraph of $G$ \emph{induced} by $U$, i.e.\ the subgraph of $G$ with vertex set $U$ and two vertices being
adjacent in $G[U]$ if and only if they are adjacent in $G$. 
We say that a graph $H$ is an \emph{induced subgraph} of $G$, or $G$ \emph{contains} $H$ as an induced subgraph, 
if $H$ is isomorphic to $G[U]$ for some $U\subseteq V(G)$. An \emph{embedding} of $H$ in $G$ is an injective map $\phi:V(H) \rightarrow V(G)$ which witnesses $H$ as an induced subgraph of $G$. Thus, $vw\in E(H)$ if and only if $\phi(v)\phi(w)\in E(G)$. By a slight abuse of notation, we will often use the term $\phi(H)$ to mean the induced subgraph of $G$ on the vertex set $\phi(V(H))$.

If no subset of $V(G)$ induces the graph $H$, we say that 
$G$ is $H$-free. The set of all $H$-free graphs is denoted by $\Free(H)$, and, by extension, if $\{H_1,H_2,\dots\}$ is a (possibly infinite) collection of graphs, denote by $\Free(H_1,H_2,\dots)$ the set of all graphs $G$ that contain none of $H_1,H_2,\dots$. 

A \emph{hereditary property} of graphs (which in this paper we will also call a \emph{graph class}) is a set of graphs $\C$ such that whenever $G\in\C$ and $H$ is an induced subgraph of $G$, then also $H\in \C$. It is well-known that for any hereditary property $\C$ there exists a unique (but not necessarily finite) set of \emph{minimal forbidden graphs}, $\{H_1,H_2,\dots\}$, such that $\C = \Free(H_1,H_2,\dots)$. If the set of minimal forbidden graphs for a class $\C$ is finite, then we say that $\C$ is \emph{finitely-defined}.

\subsection{Split permutation graphs and bichain graphs}

In this paper, we are concerned primarily with two specific classes of graphs, which in this subsection we will define and then review some of their key properties. In a graph, a \emph{clique} is a subset of pairwise adjacent vertices, and an \emph{independent set} is a subset of pairwise
non-adjacent vertices. A graph $G$ is a \emph{split} graph if its vertices can be partitioned into an independent set and a clique, 
and $G$ is a \emph{bipartite} graph if its vertices can be partitioned into two independent sets (also called colour classes or
simply \emph{parts}).

Our first class is the class of \emph{split permutation graphs}. As the name suggests, these are split graphs that also happen to be permutation graphs, although we will not use this characterisation in the sequel.  Instead, it was shown in~\cite{korpelainen:split:} that the class of split permutation graphs is the same as the class of split graphs having Dilworth number at most two, where the Dilworth number
of a graph $G$ is the size of a largest antichain with respect to the following quasi-order $\preceq$ defined on the vertices of $G$:
\[
x\preceq y \mbox{ if and only if } N_G(x)\subseteq N_G(y)\cup \{y\}.
\]
Importantly, we have the following fact.

\begin{lemma}[Benzaken, Hammer and de Werra~\cite{benzaken:split-graphs:}]The split graphs of Dilworth number at most two are characterised by seven minimal forbidden induced subgraphs.\end{lemma}

For completeness, the forbidden induced subgraphs of split graphs with Dilworth number at most two (and hence split permutation graphs\footnote{Note that it is possible to compute the forbidden induced subgraphs of split permutation graphs directly by intersecting the class of split graphs with the class of permutation graphs, both of whose sets of minimal forbidden induced subgraphs are well-known.}) are $2K_2, C_4, C_5, \text{Sun}_3, \text{Co-sun}_3, \text{Rising-sun}$ and $\text{Co-rising-sun}$. 

The second class we will study are the bichain graphs. We say that a set of vertices forms a \emph{chain} if their neighbourhoods
form a chain with respect to set inclusion, i.e.\ if they can be 
linearly ordered under the set inclusion relation.
A bipartite graph will be called a \emph{$k$-chain graph} if the vertices in each part of its bipartition can be divided into at most $k$ chains. 

The 1-chain graphs are known simply as \emph{chain graphs}, and a typical example
of a chain graph is illustrated in Figure~\ref{fig:H5}. The importance of this 
example is due to the fact that the represented graph contains as induced subgraphs all chain graphs 
with at most 5 vertices, i.e.\ it is `5-universal'. This concept of `universality' will be crucial later, and is introduced formally in Section~\ref{sec:universal}.
\begin{figure}[ht]
\centering
\begin{tikzpicture}[scale=1.8]
\foreach \j in {5,...,1}
{
\node (\j1) at (\j,1) {};
\node (\j2) at (\j,2) {};
\foreach \k in {\j,...,5}
\draw (\k1) -- (\j2);
}
\end{tikzpicture}
\caption{5-universal chain graph}
\label{fig:H5}
\end{figure}
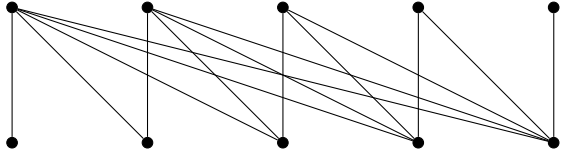

The simple structure of chain graphs implies many nice properties. In particular, 
the clique-width of chain graphs is at most three \cite{golumbic:on-the-clique-width:} and they are well-quasi-ordered by 
the induced subgraph relation~\cite{petkovsek:letter-graphs-a:}. 

\medskip
The class of $2$-chain graphs, also called \emph{bichain graphs}, 
has a much richer theory, and is our second class of study in this paper. We note here the following result from Korpelainen's PhD thesis.

\begin{prop}[{Korpelainen~\cite[Proposition 3.3.10]{korpelainen:phd}}]
The class of bichain graphs\footnote{Note that Korpelainen~\cite{korpelainen:phd} uses the term `double bichain' to mean `bichain' as defined here.} is equal to the class of $(P_7, C_6, 3K_2)$-free bipartite graphs.
\end{prop} 

Because the graph $P_7$ excludes odd cycles on 9 or more vertices, we note that (by adding the cycles $C_3, C_5$ and $C_7$ to the list of forbidden graphs) the result above implies that the class of bichain graphs is finitely-defined.

We have shown that both the class of bichain graphs and the class of split permutation graphs are finitely-defined, but in fact, these classes (and their sets of minimal forbidden graphs) are closely connected -- from the definition of bichain graphs and the characterization of split permutation graphs in terms of their Dilworth number we have 
the following.

\begin{prop}[See Korpelainen~\cite{korpelainen:phd}]\label{prop:sp-bi}
Let $G$ be a split graph given together with a partition of its vertex
set into a clique $C$ and an independent set $I$, and let $G^*$ be 
the bipartite graph obtained from $G$ by deleting the edges of $C$. 
Then  $G$ is a split permutation graph if and only if $G^*$ is 
a bichain graph. 
\end{prop}   

The validity of the proposition can be seen by noting that when $G$ is a split graph that is also a forbidden induced subgraph for the split permutation graphs, then $G^*$ is a bipartite forbidden induced subgraph for the bichain graphs, and vice-versa: if $G$ is Sun$_3$, then $G^*$ is $3K_2$, if $G$ is Co-sun$_3$, then $G^*$ is
$C_6$ and that if $G$ is Rising-sun or Co-rising-sun, then $G^*$ is $P_7$.

\subsection{Clique-width and rank-width}

In this subsection, we will introduce two graph parameters: clique-width, and its close relation rank-width. In this paper, we will not actually require the formal definitions of either of these two concepts, so we will introduce them only very briefly.

Let $G$ be a graph. The \emph{clique-width} of $G$, denoted $\cwd(G)$, is the size of the smallest alphabet $\Sigma$ such that $G$ can be constructed as a (labelled) graph using four operations: 
(1) adding a new vertex labelled by $i\in\Sigma$; 
(2) adding an edge between every vertex labelled $i$ and every vertex labelled $j$ (for distinct $i,j\in\Sigma$);
(3) for $i,j\in\Sigma$, giving all vertices labelled $i$ the label $j$; and
(4) taking the disjoint union of two previously-constructed ($\Sigma$-labelled) graphs. As alluded to in the introduction, the significance of clique-width is in the study of algorithms. Specifically, Courcelle, Makowsky and Rotics~\cite{courcelle:linear-time-sol:} showed that a large number of graph algorithms which are NP-hard in general can be solved in linear time for classes where all the graphs have clique-width at most some fixed $k$.  

The other parameter that we will utilise is the \emph{rank-width}, first introduced by Oum and Seymour in~\cite{oum:approximating-c:}. Unlike clique-width, it is not defined in terms of graph operations (although Courcelle and Kant\'e~\cite{courcelle:graph-operations:} have since given such a characterisation). Roughly speaking, rank-width is instead defined in terms of tree-like decompositions (called \emph{layouts}) where each edge has a \emph{width} determined by the rank of a certain submatrix of the adjacency matrix of the graph. 

Critically, from our perspective both rank-width and clique-width behave `the same', in the following sense.

\begin{prop}[Oum and Seymour~{\cite[Proposition~6.3]{oum:approximating-c:}}]\label{prop:cw-rw}For any graph $G$, \[\rwd(G)\leq\cwd(G)\leq 2^{\rwd(G)+1}-1.\]
\end{prop}

For a graph class $\C$, we say that $\C$ has \emph{bounded clique-width} (or, equivalently, \emph{bounded rank-width}) if there exists $k$ such that $\cwd(G)\leq k$ for all $G\in\C$, otherwise we say that $\C$ has \emph{unbounded clique-width}. Furthermore, if $\C$ has unbounded clique-width but every proper subclass of $\C$ has bounded clique-width, then we say that $\C$ is \emph{minimal} of unbounded clique-width.

In \cite{korpelainen:split:}, it was shown that the class of split permutation graphs has unbounded 
clique-width, and is not well-quasi-ordered by the induced subgraph relation. 
Together with Proposition~\ref{prop:sp-bi} this implies the same conclusions for 
bichain graphs. Indeed, in \cite{kaminski:recent-developm:} it was shown that a single application of the operation of 
transforming a clique into an independent set does not change 
the clique-width of the graph ``too much'', i.e.\ under this operation any class of graphs of bounded clique-width
transforms into a class of bounded clique-width, and vice-versa. 

Taking into account the relationship between these classes revealed in Proposition~\ref{prop:sp-bi}, the rest of this paper will focus on the results for bichain graphs only.  


\section{Universal graphs}\label{sec:universal}


For a graph class $\C$, we say that the graph $G\in\C$ is \emph{$n$-universal} if every $H\in\C$ with $n$ vertices is an induced subgraph of $G$. In this section, we will introduce a number of universal graphs for different graph classes.

\subsection{Bipartite permutation graphs} 
The first class for which we will exhibit a universal graph is the class of bipartite permutation graphs. Although distinct from the two classes of interest to us, it has a crucial role in our proofs.

Denote by $X_{m,n}$ the graph with $mn$ vertices arranged in $m$ columns and $n$ rows, in which any two consecutive columns induce a chain graph -- an example of the graph $X_{6,6}$ is shown on the left of Figure~\ref{fig:H}.

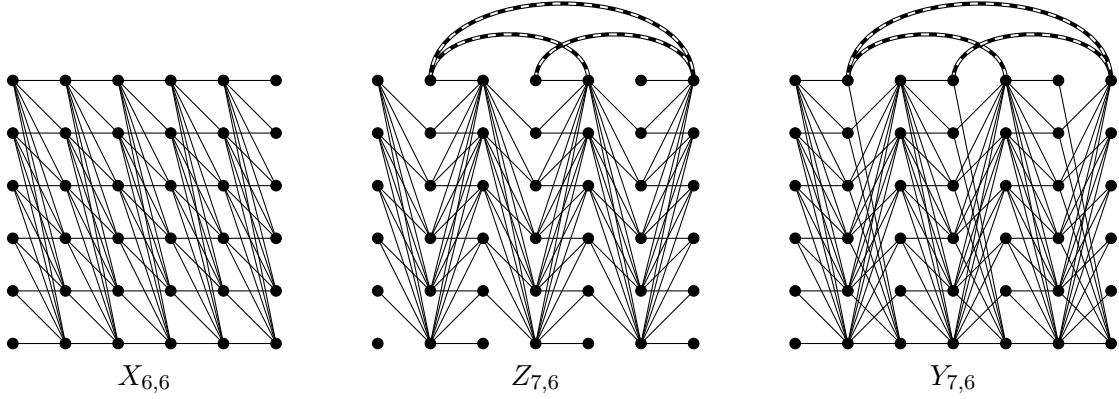
\begin{figure}[ht]
\centering
\begin{tabular}{ccccc}
\begin{tikzpicture}[scale=0.7]
\mnnodearray{6}{6}
\foreach \col in {1,...,5}
	\downhoriz{\col}{6};
\end{tikzpicture}
&\rule{10pt}{0pt}&
\begin{tikzpicture}[scale=0.7]
\curvetop256
\curvetop476
\curvetop276
\mnnodearray{7}{6}
\foreach \col in {1,3,5}
	\downnohoriz{\col}{6};
\foreach \col in {2,4,6}
	\uphoriz{\col}{6};
\end{tikzpicture}
&\rule{10pt}{0pt}&
\begin{tikzpicture}[scale=0.7]
\curvetop256
\curvetop476
\curvetop276
\mnnodearray{7}{6}
\foreach \col in {1,3,5}
	\downhoriz{\col}{6};
\foreach \col in {2,4,6}
	\upnohoriz{\col}{6};
\foreach \baserow in {3,5,7}
	\foreach \i in {1,...,6}
		\draw (\baserow,1) -- (\baserow - 1,\i);
\end{tikzpicture}\\
$X_{6,6}$&&$Z_{7,6}$&&$Y_{7,6}$
\end{tabular}
\caption{Graphs $X_{6,6}$ (left), $Z_{7,6}$ (middle), and $Y_{7,6}$ (right). The graph $Z_{7,6}$ and $Y_{7,6}$ contains the edges shown in the picture and the ``diagonal'' edges connecting every even column $i$ to every odd column $i'\ge i+3$ (these edges are represented by the curved lines at the top of  the picture).}
\label{fig:H}
\end{figure}

The following theorem was proved in \cite{lozin:minimal-universal:}.
\begin{theorem}[Lozin and Rudolf~\cite{lozin:minimal-universal:}]\label{thm:bp-universal}
The graph $X_{n,n}$ is an $n$-universal bipartite permutation graph, i.e. it contains every bipartite permutation graph
with $n$ vertices as an induced subgraph.
\end{theorem}

We call any graph of the form $X_{n,n}$ an \emph{$X$-grid}. Since the $X$-grids are universal, the following theorem allows us to conclude that the bipartite permutation graphs are minimal with unbounded clique-width.

\begin{theorem}[Lozin~\cite{lozin:minimal-classes:}]\label{thm:bp-clique-width}
For every $n$, the clique-width of $X_{n,n}$-free bipartite permutation graphs is bounded by a constant.
\end{theorem}

In our work here, we will need the following technical lemma to control how some $X$-grid can embed into a larger one. We remark that the argument used here (which employs the pigeonhole principle) may be unnecessarily wasteful, but this is of no consequence for our purposes.

\begin{lemma}\label{lem:X-embeddings}For any $n\geq1$, every embedding of the grid $X_{n,4n-1}$ into a larger $X$-grid $X_{M,N}$ contains a copy of $X_{n,n}$ that occupies exactly $n$ contiguous columns, with each column of $X_{n,n}$ embedding into a column of $X_{M,N}$.\end{lemma}

\begin{proof}
First observe that the case $n=1$ is trivial, so now we may assume that $n>1$.  Fix an arbitrary embedding of $X_{n,4n-1}$ into $X_{M,N}$ with $M\geq n$ and $N\geq 4n-1$. We first claim that the entries of any column of $X_{n,4n-1}$ must occupy at most two columns of $X_{M,N}$, and that these two columns have exactly one column separating them. 

To see this, let $u,v$ be two distinct vertices in the same column of $X_{n,4n-1}$. Observe that there exists at least one vertex $w$ in $X_{n,4n-1}$ that is adjacent to both $u$ and $v$ (for example, the vertex at the bottom of the column immediately to the right of $u$ and $v$). This tells us that in the embedding into $X_{M,N}$, the vertices $u$ and $v$ must both be placed in a column of the same parity, since $X_{M,N}$ is bipartite. Furthermore, $u$ and $v$ can have at most one column between them otherwise they can have no common neighbourhood. This proves the claim.

Thus, for any column of $X_{n,4n-1}$ we have that at least $2n$ of the vertices are embedded into the same column of $X_{M,N}$, by the pigeonhole principle. Let $v_1,\dots,v_{2n}$ be $2n$ vertices from the leftmost column of $X_{n,4n-1}$, ordered from top to bottom, that embed in one column $c_1$ of $X_{M,N}$. Note that because the neighbourhoods of $v_1,\dots,v_{2n}$ are ordered by inclusion, they must embed in $X_{M,N}$ either in the correct order, or in reverse. Appealing to the automorphism of $X$-grids obtained by rotating our standard picture by $180^\circ$, we may assume they embed in the correct order.

Now let $v_i^{(j)}$ be the vertex of $X_{n,4n-1}$ that is in the same row as $v_i$, and in the $j$th column from the left (thus $v_i = v_i^{(1)}$). Note that $v_i^{(j+1)}$ is adjacent to all of $v_1^{(j)},\dots,v_{i}^{(j)}$, and none of $v_{i+1}^{(j)},\dots,v_{2n}^{(j)}$. Furthermore, because $v_1,\dots,v_{2n}$ were embedded in $X_{M,N}$ from top to bottom, we conclude that the only possible column of $X_{M,N}$ in which $v_1^{(2)},\dots,v_{2n-1}^{(2)}$ can embed is the one immediately to the right of $c_1$, and again they must embed from top to bottom. (Note that $v_{2n}^{(2)}$ is adjacent to all of $v_1,\dots,v_{2n}$ and can thus be embedded in the column to the left of $c_1$.)

If $n=2$ we are now done, so assume $n>2$. For $i=3,\dots,n$ we follow  similar argument to $i=2$, and conclude inductively that because $v_1^{(i-1)},\dots,v_{2n-i+2}^{(i-1)}$ must embed from top to bottom in the $(i-2)$nd column to the right of $c_1$, then $v_1^{(i)},\dots,v_{2n-i+1}^i$ must embed in the $(i-1)$th column. However, this shows that the set $\{v_i^{(j)}:1\leq i,j \leq n\}$ is a copy of $X_{n,n}$ inside $X_{M,N}$ satisfying the statement of the lemma.
\end{proof}

\subsection{Bichain graphs}

For the class of bichain graphs, we will now introduce two universal constructions. 

Denote by  $Z_{n,k}$ the graph with the vertex set 
$\{z_{i,j}\ :\  1 \leq i \leq n, 1 \leq j \leq k\}$ and with $z_{i,j}z_{i',j'}$ being an edge if and only if
\begin{enumerate}
\item[(1)] $i$ is odd, $i'=i+1$ and $j > j'$,
\item[(2)] $i$ is even, $i'=i+1$ and $j \leq j'$,
\item[(3)] $i$ is even, $i'$ is odd and $i' \geq i+3$.
\end{enumerate}  

We call the edges of type (3) the \emph{diagonal} edges. An example of the graph $Z_{n,k}$ 
with $n=7$ and $k=6$ is represented in the middle picture of Figure~\ref{fig:H}, where for clarity the diagonal edges are represented by the curved lines at the top. We have: 

\begin{theorem}[Brignall, Lozin and Stacho~\cite{bls:bichain:}]
\label{thm:universal}
The graph $Z_{n,n}$ is an $n$-universal bichain graph.
\end{theorem}

Any graph  of the form $Z_{n,k}$ will be called a \emph{$Z$-grid}. 

The other universal construction we need is another, related, grid construction which we will call the \emph{$Y$-grid}. Denote by $Y_{n,k}$ the graph with the vertex set 
$\{y_{i,j}\ :\  1 \leq i \leq n, 1 \leq j \leq k\}$ and with $y_{i,j}y_{i',j'}$ being an edge if and only if
\begin{enumerate}
\item[(1)] $i$ is odd, $i'=i+1$ and $j \geq j'$,
\item[(2)] $i$ is even, $i'=i+1$ and $j < j'$,
\item[(3)] $i$ is even, $i'$ is odd and $i' \geq i+3$,
\item[(4)] $i$ is odd, $i'=i-1$ and $j=1$.
\end{enumerate}

The graph $Y_{7,6}$ is shown on the right of Figure~\ref{fig:H}. Note that the edges of type (4) simply connect the odd-column vertices in the bottom row to all the vertices in the preceding column. The reason for this anomaly will become evident in the next section, for the now we will simply refer to all the vertices that lie in row 1 of a $Y$-grid as the vertices of the \emph{bottom row}.

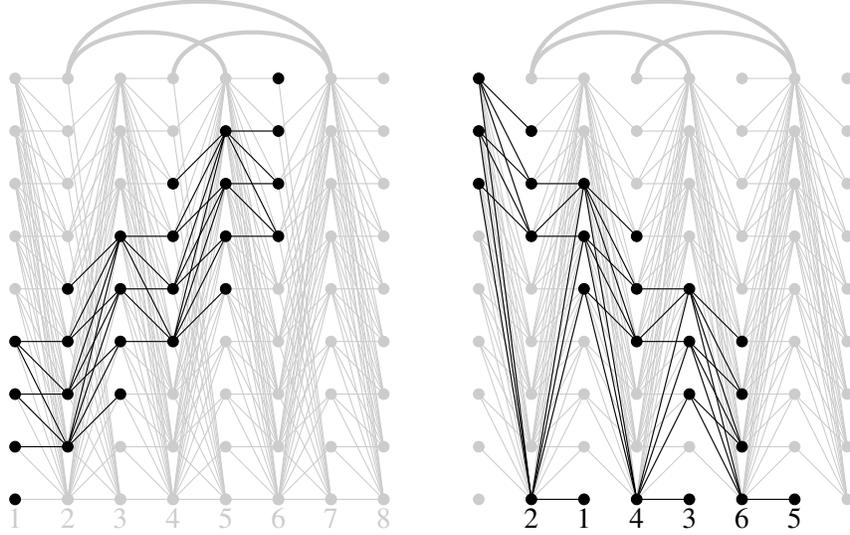
\begin{figure}[ht]
\centering
\begin{tabular}{ccc}
\begin{tikzpicture}[scale=0.7]
{%
\tikzstyle{every node}=[circle, draw=black!20, fill=black!20,
                        inner sep=0pt, minimum width=4pt]
\tikzset{every path/.append style={draw=black!20}}
\foreach \colA/\colB in {4/7,2/5,2/7}
	\curvetop{\colA}{\colB}{9};
\mnnodearray{8}{9}
\foreach \col in {1,3,5,7}
	\downhoriz{\col}{9};
\foreach \col in {2,4,6}
	\upnohoriz{\col}{9};
\foreach \i/\j in {6/7,4/5,2/3}
	\foreach \k in {1,...,9}
		\draw (\i,\k) -- (\j,1);
{\color{black!20}\labelrow{1}{8}}
}%
\foreach \i in {1,...,6}
	\foreach \j in {0,...,3}
		\node at (\i,\i+\j) {};
\foreach \i in {1,3,5}
{
	\draw (\i+1,\i+1) -- (\i,\i+2) -- (\i+1,\i+2) -- (\i,\i+3);
	\draw (\i+1,\i+3) --(\i,\i+3) -- (\i+1,\i+1) -- (\i,\i+1);
}
\foreach \i in {2,4}
{
	\draw (\i+1,\i+1) -- (\i,\i) -- (\i+1,\i+2) -- (\i,\i+1) -- (\i+1,\i+3) -- (\i,\i+2) -- (\i+1,\i+4) -- (\i,\i+3);
	\draw (\i,\i+1) -- (\i+1,\i+4) -- (\i,\i) -- (\i+1,\i+3);
}
\end{tikzpicture}
&\rule{5pt}{0pt}&
\begin{tikzpicture}[scale=0.7]
{%
\tikzstyle{every node}=[circle, draw=black!20, fill=black!20,
                        inner sep=0pt, minimum width=4pt]
\tikzset{every path/.append style={draw=black!20}}
\foreach \colA/\colB in {4/7,2/5,2/7}
	\curvetop{\colA}{\colB}{9};
\mnnodearray{8}{9}
\foreach \col in {1,3,5,7}
	\downnohoriz{\col}{9};
\foreach \col in {2,4,6}
	\uphoriz{\col}{9};
}%
\foreach \i in {1,...,6}
	\foreach \j in {0,...,2}
		\node at (\i,10-\i-\j) {};
\foreach \i in {2,4,6}
{
	\node at (\i,1) [label=below:\i] {};
	\draw (\i,1) -- (\i+1,1);
}
\foreach \i in {1,3,5}
{
	\node at (\i+2,1) [label=below:\i] {};
	\draw (\i+1,9-\i) -- (\i, 10-\i) -- (\i+1,7-\i) -- (\i,8-\i) -- (\i+1,1);
	\draw (\i, 10-\i) -- (\i+1,8-\i) -- (\i, 9-\i) -- (\i+1,7-\i);
	\draw (\i, 9-\i) -- (\i+1,1) -- (\i,10-\i);
}
\foreach \i in {2,4}
{
	\draw (\i,9-\i) -- (\i+1,9-\i) -- (\i,8-\i) -- (\i+1, 8-\i) -- (\i,1);
	\draw (\i+1,9-\i) -- (\i,1) -- (\i+1,7-\i);
}
\end{tikzpicture}
\end{tabular}
\caption{On the left, embedding $Z_{6,4}$ in $Y_{8,9}$. On the right, embedding $Y_{6,4}$ in $Z_{8,9}$.}
\label{fig:Z-Y-embeds}
\end{figure}

$Y$-grids and $Z$-grids are strongly related. Figure~\ref{fig:Z-Y-embeds} demonstrates how $Y_{8,9}$ contains an induced $Z_{6,4}$, and how $Z_{8,9}$ contains an induced $Y_{6,4}$. Indeed, this example easily generalises to the following observation.

\begin{lemma}\label{lem:ZY}
The graph $Z_{2n,2n}$ contains $Y_{n,n}$ as an induced subgraph, 
and the graph $Y_{2n,2n}$ contains $Z_{n,n}$ as an induced subgraph. 

Moreover, $Z_{n,n}$ can be embedded in $Y_{2n,2n}$ in such a way as to use none of the vertices in the bottom row.\qed
\end{lemma}

As a consequence, we note that $Y_{2n,2n}$ is an $n$-universal graph for the class of bichain graphs.


\section{Pivoting, universality and clique width}
\label{sec:main-j}
The proof of minimality of the class of bichain graphs which we present here takes an entirely different approach to that for bipartite permutation graphs in \cite{lozin:minimal-classes:}. Here, we {\it reduce} the problem to bipartite permutation graphs by means of the so-called pivoting operation.
We will show that an $X$-grid can be `pivoted' into a $Y$-grid, and when this is coupled with our Lemma~\ref{lem:X-embeddings} to control how $X$-grids can embed in larger $X$-grids, we can conclude that any proper subclass of bichain graphs must have bounded clique-width.

Let $G=(V,E)$ be a graph, and $v\in V$ a vertex. Following, Oum~\cite{oum:rank-width-and:}, define the \emph{local complementation} of $G$ at $v$ to be the graph $G*v$ with vertex set $V$ and edge set $E\triangle\{xy\ :\ xv, yv\in E\text{ and }x\neq y\}$ where $\triangle$ denotes the symmetric difference. In other words, $G*v$ is formed from $G$ by replacing the subgraph induced on the neighbourhood of $v$ with its complement. Two graphs $H$ and $G$ are \emph{locally equivalent} if $H$ can be obtained from $G$ by a sequence of local complementations.

\begin{prop}[Oum~{\cite[Corollary~2.7]{oum:rank-width-and:}}]\label{prop:rw-pivot}If $H$ is locally equivalent to $G$, then $\rwd(H)=\rwd(G)$.\end{prop}

Crucially, observe that the above proposition does not simply bound the rank-width of one graph in terms of the rank-width of the other, but actually states that they are equal. Thus, we may apply arbitrarily many such operations to a graph, and still obtain one whose rank-width is the same. Note that Gurski~\cite{gurski:the-behavior:} has recently studied the effect of various operations (including local complementation) directly on clique-width (rather than via rank-width), but the effect of a single local complementation can \emph{triple} the clique-width. 

One particularly important sequence of local complementations is the pivot:  for an edge $uv$ of $G$, the \emph{pivot} on $uv$ is the graph $G*u*v*u$. That this is well-defined (i.e.\ $u$ and $v$ can be interchanged in the definition) is established in~\cite{oum:rank-width-and:}, and the effect of this process is to complement the edges between the three sets of vertices $N(u)\cap N(v)$, $N(u)\setminus N(v)$ and $N(v)\setminus N(u)$.\footnote{Technically, pivoting on an edge of a bipartite graph also causes the labels of the two end-vertices to be interchanged, but this is of no consequence here.}

We remark further that if $G$ is a bipartite graph, then the pivot on an edge $uv$ (formed by taking the complement of the edges between $N(u)\setminus\{v\}$ and $N(v)\setminus\{u\}$) gives rise to another bipartite graph (see~\cite{oum:rank-width-and:}). 

\begin{lemma}\label{lem:pivotXY}The graph $Y_{2n,2n}$ can be obtained from $X_{2n,2n}$ by a sequence of $n$ pivots, acting from right to left on alternate edges in the bottom row of $X$.\end{lemma}

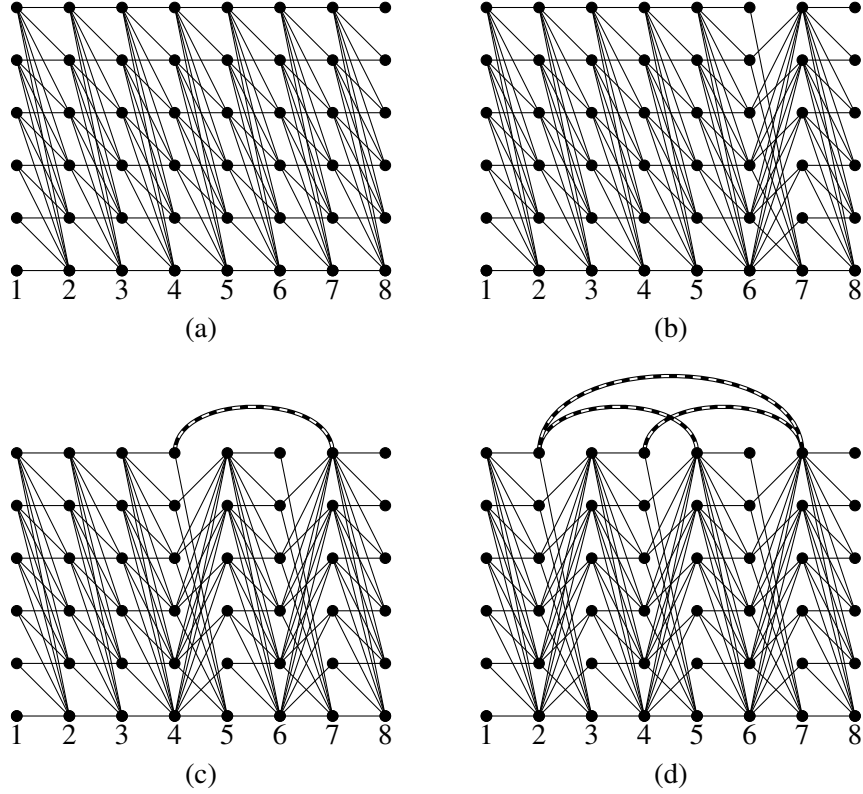
\begin{figure}[ht]
\centering
\begin{tabular}{ccc}
\begin{tikzpicture}[scale=0.7]
\mnnodearray{8}{6}
\foreach \col in {1,...,7}
	\downhoriz{\col}{6};
\labelrow{1}{8}
\end{tikzpicture}
&\rule{5pt}{0pt}&
\begin{tikzpicture}[scale=0.7]
\mnnodearray{8}{6}
\foreach \col in {1,2,3,4,5,7}
	\downhoriz{\col}{6};
\upnohoriz{6}{6}
\foreach \i/\j in {6/7}
	\foreach \k in {1,...,6}
		\draw (\i,\k) -- (\j,1);
\labelrow{1}{8}
\end{tikzpicture}
\\
(a)&&(b)\\
\begin{tikzpicture}[scale=0.7]
\curvetop{4}{7}{6}
\mnnodearray{8}{6}
\foreach \col in {1,2,3,5,7}
	\downhoriz{\col}{6};
\foreach \col in {4,6}
	\upnohoriz{\col}{6};
\foreach \i/\j in {6/7,4/5}
	\foreach \k in {1,...,6}
		\draw (\i,\k) -- (\j,1);
\labelrow{1}{8}
\end{tikzpicture}
&&
\begin{tikzpicture}[scale=0.7]
\foreach \colA/\colB in {4/7,2/5,2/7}
	\curvetop{\colA}{\colB}{6};
\mnnodearray{8}{6}
\foreach \col in {1,3,5,7}
	\downhoriz{\col}{6};
\foreach \col in {2,4,6}
	\upnohoriz{\col}{6};
\foreach \i/\j in {6/7,4/5,2/3}
	\foreach \k in {1,...,6}
		\draw (\i,\k) -- (\j,1);
\labelrow{1}{8}
\end{tikzpicture}
\\
(c)&&(d)
\end{tabular}
\caption{The pivoting sequence from $X_{8,6}$ (diagram (a)) to $Y_{8,6}$ (diagram (d)). From (a) to (b), the pivot is on edge 7--8. From (b) to (c), the pivot is on edge 5--6, and from (c) to (d), pivot on edge 3--4.}
\label{fig:pivot-from-X}
\end{figure}

\begin{proof}
This proof is illustrated in Figure~\ref{fig:pivot-from-X}. For ease of reference, let the sets of vertices contained in the columns of $X_{2n,2n}$ be $C_1,\dots,C_{2n}$, working from left to right, and label the vertices in the bottom row $x_1,\dots,x_{2n}$. We will pivot on the edges $x_{2n-1}x_{2n},x_{2n-3}x_{2n-2},\dots,x_1x_2$ in that order. 

First, note that $N(x_{2n}) = C_{2n-1}$, and $N(x_{2n-1}) = C_{2n-2} \cup \{x_{2n}\}$. Thus the effect of pivoting on $x_{2n-1}x_{2n}$ is to complement the edges between $C_{2n-2}$ and $C_{2n-1}\setminus \{x_{2n-1}\}$.

We now claim, by induction, that after the pivot on the edge $x_{2i-1}x_{2i}$, we take the complement of the edges between $C_{2i-2}$ and $C_{2i-1}\setminus\{x_{2i-1}\}$, and every vertex in column $C_{2i-2}$ becomes adjacent to every vertex in column $C_{2j-1}$ for all $j=i+1,\dots,n$. 

The base case is the edge $x_{2n-1}x_{2n}$ mentioned above, so now consider the pivot on $x_{2i-1}x_{2i}$ for some $i<n$. Since $x_{2i}\in C_{2i}$, by induction before pivoting we have $N(x_{2i}) = C_{2i-1}\cup C_{2i+1} \cup \cdots \cup C_{2n-1}$, and $N(x_{2i-1}) = C_{2i-2}\cup\{x_{2i}\}$. Before pivoting on $x_{2i-1}x_{2i}$, the only edges between these two neighbourhoods is the chain graph between $C_{2i-2}$ and $C_{2i-1}$, so after pivoting we obtain the edges required by the inductive hypothesis. 

Now, by inspection, the graph obtained after pivoting all $n$ edges $x_{2i-1}x_{2i}$ on the bottom row of $X_{2n,2n}$ is precisely $Y_{2n,2n}$.
\end{proof}

Note that the action of pivoting is an involution, thus the above proof also shows that we may pivot from $Y_{2n,2n}$ to $X_{2n,2n}$ by using alternating edges in the bottom row of the $Y$-grid, but working from left to right.

We are now ready to prove the first part of Theorem~\ref{thm:main}.

\begin{theorem}\label{thm:minimal}
The class of bichain graphs is a minimal hereditary class of unbounded clique-width.
\end{theorem}

\begin{proof}
First, note that, by Proposition~\ref{prop:cw-rw}, a class of graphs has unbounded clique-width if and only if it has unbounded rank-width. Thus it suffices to prove that the bichain graphs are minimal of unbounded rank-width.

Next, by the comment made at the end of Section~\ref{sec:pre}, the class of bichain graphs has unbounded clique- (and rank-)width, since the class of split permutation graphs does. 

Now consider any proper subclass $\C$ of bichain graphs, formed by forbidding at least one bichain graph $H$, which we may assume has $k$ vertices. According to Theorem~\ref{thm:universal},
any bichain graph with $k$ vertices can be embedded into the $Z$-grid $Z_{k,k}$. In particular, since $H$ is a forbidden graph of $\C$ with $k$ vertices, the subclass $\C$ does not contain the universal graph $Z_{k,k}$.

For each $G\in\C$ on $n$-vertices, fix some arbitrary embedding $\phi_G$ of $G$ into the $n$-universal bichain graph $Y_{2n,2n}$. With respect to $\phi_G$, now form the graph $G^+$ which comprises the induced copy of $G$ together with the set $R$ of at most $2n$ further vertices so that $G^+$ contains every vertex from the bottom row of $Y_{2n,2n}$. We will call these $2n$ vertices the \emph{pivot row} of $G^+$. Note that $G^+$ is still a bichain graph. Finally, form $G^*$ by pivoting on every other edge of the pivot row of $G^+$, as per Lemma~\ref{lem:pivotXY}, and note that $G^*$ is a bipartite permutation graph. By Proposition~\ref{prop:rw-pivot}, we have $\rwd(G^+)=\rwd(G^*)$. 

Denote by $\C^*$ the set of graphs $G^*$ that we obtain in this way, i.e.\ $\C^* = \{ G^* : G\in \C\}$. Note that $\C^*$ is a subset of the bipartite permutation graphs, but it need not be a hereditary class because the embeddings of $G$ into $Y_{2n,2n}$ were fixed arbitrarily. 

We claim that no graph in $\C^*$ contains the universal bipartite permutation graph $X_{8k-1,2k}$ as an induced subgraph. This will complete the proof, as then we have shown that every graph $G^*$ belongs to a proper subclass of the bipartite permutations graphs, and therefore there is an absolute bound on the rank-width of any such $G^*$. Since $\rwd(G^*)=\rwd(G^+)$ and $G$ is an induced subgraph of $G^+$, we conclude that $G$ must also have bounded rank-width, as required.

For a contradiction, suppose that there exists  $G^*\in\C^*$ that contains $X_{8k-1,2k}$ as an induced subgraph. Now the graph $G^*$ is a bipartite permutation graph, and it comes with an implicit embedding into the $X$-grid $X_{2n,2n}$, inherited from the fixed embedding $\phi_G$ of $G$ into $Y_{2n,2n}$, together with the set $R$ of vertices from the bottom row of $Y_{2n,2n}$. From this embedding of $G^*$ in $X_{2n,2n}$, we restrict to the vertices of $G^*$ that form an induced $X_{8k-1,2k}$, to obtain an embedding of $X_{8k-1,2k}$ in $X_{2n,2n}$. By Lemma~\ref{lem:X-embeddings} this embedding contains a copy of $X_{2k,2k}$ that lies in $2k$ consecutive columns, with $2k$ vertices in each column. 

We now pivot $X_{2n,2n}$ to $Y_{2n,2n}$ via the process in Lemma~\ref{lem:pivotXY}, and observe that the embedded copy of $X_{2k,2k}$ gets pivoted to a graph $H$ that has $2k$ vertices in each of $2k$ consecutive columns of $Y_{2n,2n}$, and which is an induced subgraph of $G^+$. Note that by construction, this graph $H$ is either isomorphic to $Y_{2k,2k}$, or it is isomorphic to an induced subgraph of $Y_{2k+1,2k+1}$. (This latter case arises if not all vertices of the pivot row of $Y_{2k,2k}$ are present in $H$, and/or if the leftmost column of $H$ embeds into an even-numbered column of $Y_{2n,2n}$.) 

Recalling that $R$ denoted the extra vertices that were added to the embedded copy of $G$ to form $G^+$, we can now find an induced subgraph $H^-$ of $G$ by removing from $H$ all vertices that were embedded into vertices from $R$. Using the second part of Lemma~\ref{lem:ZY}, inside $H$ we can find a copy of $Z_{k,k}$ that does not use any vertices of the pivot row of $Y_{2k,2k}$, and therefore we can also find $Z_{k,k}$ inside $H^-$. However, this implies that $Z_{k,k}$ is an induced subgraph of $G$, which implies $G\not\in\C$, a contradiction.
\end{proof}


\section{Antichains and well-quasi-ordering}
\label{sec:canonical}


In this section we will show that the class of bichain graphs contains a canonical antichain with respect to the labelled induced subgraph relation, and then we will remark on how this argument can be modified to yield a canonical antichain for split permutation graphs. 

We begin by defining the antichain. First, we define a sequence of graphs $S_k$ as follows. $S_k$ has vertex set $V(S_k)=\{x_i, y_i: 1\leq i \leq k\}$, and edge set defined by $x_iy_j \in E(S_k)$ if and only if $j=i$ or $j \geq i+2$. The graph $S_5$ is pictured on the left in Figure~\ref{fig:antichain}.

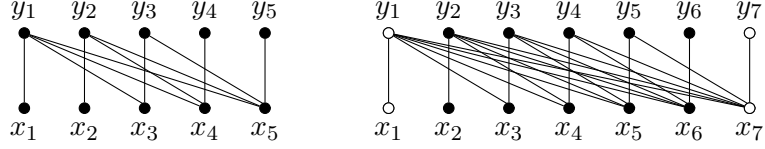
\begin{figure}[h]
\centering
\begin{tabular}{ccc}
\begin{tikzpicture}[xscale=0.8]
\foreach \i in {1,2,...,5}
{
\node (x\i) at (\i,1) [label=above:$y_{\i}$]{};
\node (y\i) at (\i,0) [label=below:$x_{\i}$]{};
\draw (x\i) -- (y\i);
}
\foreach \i in {1,2,3}
{
\pgfmathsetmacro\jstart{\i + 2};
\foreach \j in {\jstart,...,5}
\draw (x\i) -- (y\j);
}
\end{tikzpicture}
&\rule{10pt}{0pt}&
\begin{tikzpicture}[xscale=0.8]
\foreach \i in {1,2,...,7}
{
\node (x\i) at (\i,1) [label=above:$y_{\i}$]{};
\node (y\i) at (\i,0) [label=below:$x_{\i}$]{};
\draw (x\i) -- (y\i);
}
\foreach \i in {1,2,...,5}
{
\pgfmathsetmacro\jstart{\i + 2};
\foreach \j in {\jstart,...,7}
\draw (x\i) -- (y\j);
}
\node [fill=white] at (1,0) {};
\node [fill=white] at (1,1) {};
\node [fill=white] at (7,0) {};
\node [fill=white] at (7,1) {};
\end{tikzpicture}
\end{tabular}
\caption{On the left, the graph $S_5$. On the right, the 2-coloured graph $S_7^\circ$.}
\label{fig:antichain}
\end{figure}
 
First observe that the graphs $S_k$ are all bichain graphs, since the induced subgraph on the odd vertices is a chain graph, as is the induced subgraph on the even vertices. Next, notice that $S_k$ embeds into $S_{k'}$ whenever $k\leq k'$, but it must map into a consecutive set of pairs of vertices $x_i$, $y_i$. This is because the only induced copies of $2K_2$ in the graph are formed by the vertices in two consecutive columns.

Consequently, we can form an infinite antichain with two labels from $\{S_k\}$ as follows: from each graph $S_k$, form a 2-coloured graph $S_k^\circ$ by colouring the vertices $x_1$, $y_1$, $x_k$ and $y_k$ of $S_k$ white, and all remaining vertices black. The graph $S_7^\circ$ is illustrated in Figure~\ref{fig:antichain}. Without loss of generality we can assume that black vertices cannot be embedded in white vertices (otherwise we can swap the roles of black and white), so by the observation about embedding copies of $2K_2$, it follows that $\{S_k^\circ\}$ is an infinite antichain in the labelled induced subgraph ordering.

It remains to prove that $\{S_k^\circ\}$ is a \emph{canonical} labelled antichain. In other words, we need to prove that every subclass of bichain graphs containing only finitely many graphs $S_k$ is well-quasi-ordered by the labelled induced subgraph relation. 

\subsection{Structure and well-quasi-ordering}

Before we proceed to show that $\{S_k^\circ\}$ is a canonical labelled antichain in the class of bichain graphs, we require a number of concepts relating to structure and well-quasi-ordering from the literature, which we will briefly review here.

The tool we will use to prove well-quasi-orderability is the notion of \emph{letter graphs}. For $k\geq 1$, fix an alphabet $X$ of size $k$ (for example, $X=\{1,2,\dots,k\}$). A $k$-\emph{letter graph} $G$ is a graph defined by a finite word $x_1x_2\cdots x_n$, with $x_i\in X$ for all $i$, together with a subset $S\subseteq X\times X$ such that:
\begin{itemize}
\item $V(G) = \{1,2,\ldots,n\}$
\item $E(G) = \{ij : i\leq j$ and $(x_i, x_j) \in S\}$
\end{itemize}

The importance of this notion is due to the following theorem:

\begin{theorem}[Petkov\v{s}ek~\cite{petkovsek:letter-graphs-a:}]\label{thm:letter11}
For any fixed $k$, the class of $k$-letter graphs is  well-quasi-ordered by the labelled induced subgraph relation. 
\end{theorem}

When combined with the following observation, we have that any collection of graphs which can be embedded in a $Z$-grid with a fixed number of columns $k$ is well-quasi-ordered with respect to the labelled induced subgraph relation.

\begin{lemma}\label{lem:letter11}
For any $n, k \in \mathbb{N}$, $Z_{n,k}$ is a $k$-letter graph.
\end{lemma}

\begin{proof}
Let $X=\{1,2,\dots,k \}$, and define three sets of relations from $X\times X$:
\begin{itemize}
\item $R_1=\{(2i-1,2i): i=1,2,\dots,\lfloor \frac {k}{2}\rfloor\}$, 
\item $R_2 = \{(2i+1,2i): i = 1,2,\dots,\lceil \frac {k}{2}\rceil-1\}$, 
\item $D=\{(2i,2j+1), (2j+1,2i): 1 \leq i<j \leq \lceil \frac{k}{2} \rceil -1 \}$.
\end{itemize}
By inspection, with $S=R_1\cup R_2\cup D$ the $k$-letter graph associated with the word $w=(k\,(k-1)\cdots 1)^n$ is isomorphic to $Z_{n,k}$, where the letters of the word $w$ correspond to the vertices of the $Z$-grid. See Figure~\ref{fig:letter} for the case $k=7$.  
\end{proof}

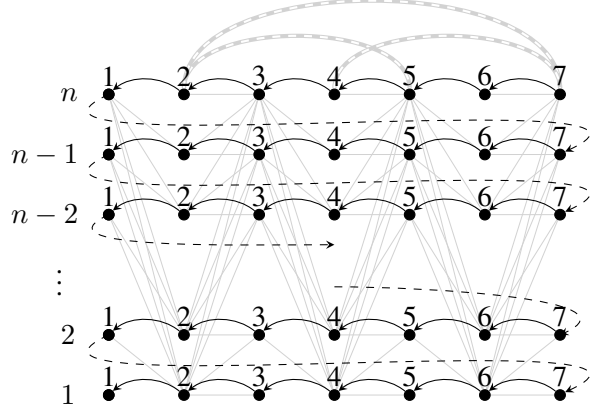
\begin{figure}[h]
\centering
\begin{tikzpicture}[>=stealth,yscale=0.8]
\foreach \j in {6,5,4,2,1}
{
\node[lightgray!70] (1-\j) at (1,\j) {};
\foreach \i in {2,4,6}
{
\pgfmathtruncatemacro\im{\i -1};
\pgfmathtruncatemacro\ip{\i +1};
\node[lightgray!70] (\i-\j) at (\i,\j) {};
\node[lightgray!70] (\ip-\j) at (\ip,\j) {};
\foreach \k in {\j,...,6}
{
\ifthenelse{\NOT \j = 3 \AND \NOT \k = 3}{
\draw[ultra thin,lightgray!70] (\ip-\k) -- (\i-\j);
\ifthenelse{\NOT \j = \k}{\draw[ultra thin,lightgray!70] (\im-\k) -- (\i-\j);}{}}{}
}
}
}
\foreach \a/\b in {2-6/5-6,2-6/7-6,4-6/7-6}
{
\draw[ultra thick,lightgray!70] (\a) edge[out=90,in=90] (\b);
\draw[thick,white] (\a) edge[out=90,in=90,dashed] (\b);
}
\draw (1,4) edge[out=-150,in=180,dashed,->] (4,3.5);
\draw (4,2.8) edge[out=0,in=30,dashed,->] (7,2);
\foreach \j in {6,5,4,2,1}
{
\node (7\j) at (7,\j) [label=above:7] {};
\pgfmathtruncatemacro\jp{\j + 1};
\ifthenelse{\NOT \j = 6 \AND \NOT \j=2}{\draw (1,\jp) edge[out=-150,in=30,dashed,->] (7\j);}{}
\foreach \i in {6,...,1}
{
\node (\i\j) at (\i,\j) [label=above:\i] {};
\pgfmathtruncatemacro\ip{\i + 1};
\draw[->] (\ip\j) edge[out=135,in=45] (\i\j);
}
}
\foreach \j/\l in {6/n,5/n-1,4/n-2,2/2,1/1}
\node (1\j a) at (1,\j) [label={[label distance=3mm]left:$\l$}] {};
\node [white] (dots) at (1,3) [label={[label distance=3mm]left:$\vdots$}] {};
\end{tikzpicture}
\caption{Reading $Z_{n,7}$ as a $7$-letter graph}
\label{fig:letter}
\end{figure}

In order to apply Theorem~\ref{thm:letter11} and Lemma~\ref{lem:letter11}, we require the following structural concept. For any two disjoint bipartite graphs 
$G_1=(X_1,Y_1,E_1)$ and $G_2=(X_2,Y_2,E_2)$ define the following three binary operations:
\begin{description}
\item[disjoint union] $G_1\oplus G_2=(X_1\cup X_2,\ Y_1\cup Y_2,\ E_1\cup E_2)$;
\item[join] $G_1\otimes G_2 = (X_1\cup X_2,\ Y_1\cup Y_2,\ E_1\cup E_2\cup (X_1\times Y_2) \cup (X_2\times Y_1))$\newline
(that is, the bipartite complement of the disjoint union of the bipartite complements of $G_1$ and $G_2$);
\item[skew join]
$G_1\oslash G_2=(X_1\cup X_2,\ Y_1\cup Y_2,\ E_1\cup E_2\cup (X_1\times Y_2))$.
\end{description}

These three operations define a decomposition scheme known as the \emph{canonical decomposition}, which takes a bipartite graph $G$ and whenever $G$ has one of the following three forms 
$G=G_1\oplus G_2$, $G=G_1\otimes G_2$, or $G=G_1\oslash G_2$,
partitions it into $G_1$ and $G_2$, and then the scheme applies to $G_1$ and $G_2$ recursively. 

Graphs that cannot be decomposed into smaller graphs under this scheme are called \emph{canonically prime}. The following theorem allows us to restrict our attention from now on to canonically prime graphs only.

\begin{theorem}[Atminas and Lozin~\cite{atminas:labelled-induced:}]\label{thm:canonical}
Let $\C$ be a hereditary class of bipartite graphs. If the canonically prime graphs in $\C$ are well-quasi-ordered by the labelled induced subgraph relation, then all graphs in $\C$ are well-quasi-ordered by the labelled 
induced subgraph relation.  
\end{theorem}

\subsection{The canonical antichain}

We now have the concepts from the literature that we need to prove that $\{S_k^\circ\}$ is canonical. Our task now is to show that the canonically prime graphs in a subclass of bichain graphs with only finitely many of the (unlabelled) graphs $\{S_k\}$ are embeddable into $Z_{n,k}$ for some $k$. 

Every vertex $v\in V(Z_{n,k})$ can be represented by its row/column coordinates, where row 1 is the bottom row, and column 1 is the leftmost row. For notational convenience in our proofs, define $\row(v)$ to mean the index of the row of $Z_{n,k}$ in which $v$ lies, and $\col(v)$ to mean the index of the column of $Z_{n,k}$ containing $v$.

We begin with a useful observation about how canonically prime graphs can embed into a $Z$-grid.

\begin{lemma}\label{lem:nogaps}
Let $G$ be a canonically prime bichain graph, and $\phi$ an embedding of $G$ into $Z_{n,k}$, for some $n,k$. Then $\phi(G)$ occupies a consecutive set of columns of $Z_{n,k}$.
\end{lemma}

\begin{proof}
Suppose that there exists a column in $Z_{n,k}$ that contains no vertices of $\phi(G)$, and suppose the index of this column is $c$. Now let $L$ be the induced subgraph on all vertices $v\in \phi(G)$ for which $\col(v)<c$, and $R$ the induced subgraph on all vertices $v\in \phi(G)$ for which $\col(v)>c$. Note that every $v\in\phi(G)$ is in exactly one of $L$ or $R$.

By inspection, except for the immediately preceding column, the vertices in any given column of index $c$ of the grid $Z_{n,k}$ are adjacent either to none of the vertices in the columns to the left of it (when $c$ is even), or they are adjacent to all even-indexed, and no odd-indexed columns to the left. From this, the embedding $\phi(G)$ tells us that $G=L\oslash R$ is a skew join, which is impossible unless one of $R$ or $L$ contains no vertices, completing the proof. 
\end{proof}

Our next lemma is the crucial step to understand the relationship between canonically prime graphs and the graphs $S_k$. 

\begin{lemma}\label{lem:induced-sk}Let $G\in \Free(S_{k},S_{k+1},\dots)$ be a canonically prime bichain graph. Then $G$ can be embedded in $Z_{n,2k}$ for some $n$.\end{lemma}

\begin{proof}
Suppose that $G$ does not embed in $Z_{n,2k}$ for any $n$. Aiming for a contradiction, we will show that $G$ contains $S_{k}$. 

Let $m=|G|$. Since $G$ is a bichain graph, by Theorem~\ref{thm:universal} it must embed in the $m$-universal graph $Z_{m,m}$. By Lemma~\ref{lem:nogaps}, any such embedding must occupy a consecutive set of columns, since $G$ is canonically prime. Moreover, by our assumption that $G$ does not embed in $Z_{n,2k}$, every embedding of $G$ into $Z_{m,m}$ must occupy at least $2k+1$ consecutive columns. Pick one such embedding $\phi$, and note that we  may assume that $\phi(G)$ contains at least one vertex in the first (i.e.\ leftmost) or second column of $Z_{m,m}$.

We now choose a sequence of vertices $v_1,v_2,v_3,\dots v_{2k}$ from $\phi(G)$ as follows. We proceed inductively, at each step choosing $v_i$ so that $\col(v_i) \leq i+1$. If $\phi(G)$ contains a vertex in the leftmost column of $Z_{m,m}$, then set $v_1$ to be the highest vertex in this column. Otherwise, set $v_1$ to be the highest vertex in the third column of $Z_{m,m}$. In either case, there are still at least $2k-1$ contiguous columns containing vertices from $\phi(G)$ lying in columns with index greater than $\col(v_1)$.

Next, suppose $v_1,\dots,v_{i}$ have been chosen for some $1\leq i< 2k$, with $\col(v_j) \leq j+1$ for all $j=1,\dots,i$. The following argument is illustrated in Figure~\ref{fig:embeddingSinZ}. Define \[A_i = \{v\in\phi(G):\col(v) = \col(v_i)+1 \text{ and } \row(v)<\row(v_i)\},\] that is, the set of all vertices of $\phi(G)$ in the column to the right of $v_i$, and in a row strictly below $v_i$. If $A_i\neq \emptyset$, then choose $v_{i+1}$ from $A_i$ so that $\row(v_{i+1})$ is maximal.

If $A_i=\emptyset$, then the set $B_i=\{v\in\phi(G):\col(v) = \col(v_i)+1 \text{ and } \row(v)\geq\row(v_i)\}$ must be non-empty, since $\phi(G)$ has at least one vertex in the column with index $\col(v_{i})+1$ (since  $\col(v_{i})\leq i+1 < 2k+1$). Pick $b_i \in B_i$ so that $\row(b_i)$ is minimal. Now define
\[C_i = \{v\in\phi(G):\col(v) = \col(b_i)-1 \text{ and } \row(v)>\row(b_i)\},\]
that is, all vertices of $\phi(G)$ in the column to the left of $b_i$ but above $b_i$. Note that $C_i$ must be non-empty, otherwise $G$ is a skew join: this can be seen by partitioning the vertices of $\phi(G)$ into two, namely all vertices in columns with indices at most $\col(v_i)$, and all those in columns with indices strictly greater than $\col(v_i)$. Thus, pick $c_i\in C_i$ so that $\row(c_i)$ is minimal.

Next, define
\[D_i = \{v\in\phi(G):\col(v) = \col(c_i)-1 \text{ and } \row(v)>\row(c_i)\}.\]
By a similar argument to that for $C_i$, observe that $D_i$ must be non-empty, otherwise $G$ is a skew join between vertices which $\phi$ maps to $C_i$ or columns to the right of $C_i$, and all other vertices. Thus, pick $d_i\in D_i$ so that $\row(d_i)$ is minimal.

Similarly, let 
\[E_i = \{v\in\phi(G):\col(v) = \col(d_i)-1 \text{ and } \row(v)>\row(d_i)\},\]
which must be non-empty (consider the set of vertices in $D_i\cup C_i\cup \{\text{all vertices further right}\}$), pick $e_i\in E_i$ so that $\row(e_i)$ is minimal. Finally, define
\[F_i = \{v\in\phi(G):\col(v) = \col(e_i)-1 \text{ and } \row(v)>\row(e_i)\},\]
which must be non-empty (consider $E_i\cup D_i\cup C_i\cup \{\text{all vertices further right}\}$). We now pick $v_{i+1}\in F_i$, $v_{i+2}\in E_i$, $v_{i+3}\in D_i$, $v_{i+4}\in C_i$ and $v_{i+5}\in B_i$ in turn, each chosen as high as possible while still satisfying the condition $\row(v_{i+1})>\row(v_{i+2})>\row(v_{i+3})>\row(v_{i+4})>\row(v_{i+5})$. (If $i+5 > 2k$ then we can simply stop once we have picked $v_{2k}$.) Note that the case $A_i=\emptyset$ can in fact only happen when $i\geq 5$, otherwise $G$ fails to be canonically prime or $v_1$ was not chosen correctly.

\begin{figure}
\centering
\begin{tikzpicture}[yscale=0.4,xscale=0.8,
					every label/.append style={rectangle}]
					
{
\tikzstyle{every node}=[circle, draw=black!20, fill=black!20,
                        inner sep=0pt, minimum width=4pt]
\tikzset{every path/.append style={draw=black!20}}
\mnnodearray{9}{16}
\foreach \col in {1,3,5,7}
	\downnohoriz{\col}{16};
\foreach \col in {2,4,6,8}
	\uphoriz{\col}{16};
\foreach \i/\j in {2/5,2/7,2/9,4/7,4/9,6/9}
	\curvetop{\i}{\j}{16};
}
\foreach \x/\y/\l in {4/6/$v_{i-2}$,4/15/$v_{i+2}$,6/10/$v_{i+4}$,7/8/$v_{i+5}$}
\node (\x-\y-a) at (\x,\y) [label={[label distance=0.8mm]right:{\tiny\l}}] {};
\foreach \x/\y/\l in {2/9/$v_{i-4}$,3/7/$v_{i-3}$,5/5/$v_{i-1}$,6/3/$v_i$,3/16/$v_{i+1}$,5/11/$v_{i+3}$}
\node (\x-\y-a) at (\x,\y) [label={[label distance=0.8mm]left:{\tiny\l}}] {};
\foreach \a/\b in {2-9/3-16,3-7/4-6,3-16/4-6,3-16/4-15,4-6/5-11,5-5/6-3,5-11/6-3,5-11/6-10,6-3/7-8}
\draw [thick] (\a-a) -- (\b-a);
\path[draw, thin, rounded corners, gray] (6.75,2.5) rectangle (7.25,0.5);
\path[draw, thin, rounded corners, gray] (6.75,2.5) rectangle (7.25,16.5);
\path[draw, thin, rounded corners, gray] (5.75,6.5) rectangle (6.25,16.5);
\path[draw, thin, rounded corners, gray] (4.75,8.5) rectangle (5.25,16.5);
\path[draw, thin, rounded corners, gray] (3.75,12.5) rectangle (4.25,16.5);
\path[draw, thin, rounded corners, gray] (2.75,14.5) rectangle (3.25,16.5);
\foreach \x/\y/\l in {7/0/$A_i$,7/17/$B_i$,6/17/$C_i$,5/17/$D_i$,4/17/$E_i$,3/17/$F_i$}
\node[draw=none, fill=none] at (\x,\y) {\l}; 
\foreach \x/\y/\l in {7/3/$b_i$,6/7/$c_i$,5/9/$d_i$,4/13/$e_i$}
\node [gray] (\x-\y-a) at (\x,\y) [label={[label distance=0.1mm]91:{\tiny\l}}] {};
\end{tikzpicture}
\caption{The description of the sets $A_i$, $B_i$, $C_i$, $D_i$, $E_i$ and $F_i$ in the proof of Lemma~\ref{lem:induced-sk}, and the related vertices when $A_i=\emptyset$.}
\label{fig:embeddingSinZ}
\end{figure}
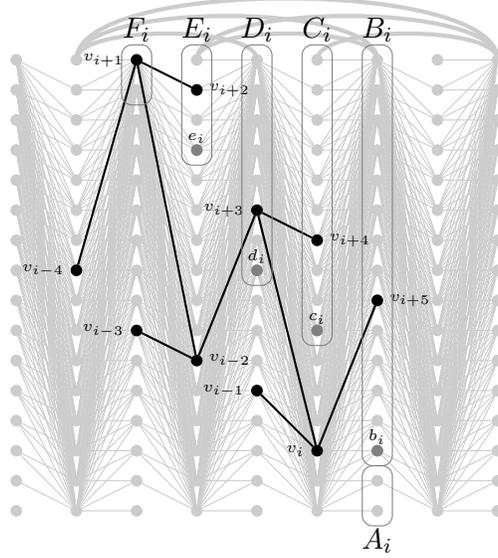

By inspection, we now observe that the vertices $v_1,\dots, v_{2k}$ induces a copy of $S_k$, by setting $x_i = v_{2i-1}$ and $y_i = v_{2i}$ for $i=1,2,\dots, k$. Thus $G$ contains $S_k$, which yields the desired contradiction.  
\end{proof}

We are now in a position to state and prove the main result of this section.

\begin{theorem}
The antichain $\{S_k^\circ\}$ is a canonical antichain for the class of bichain graphs, under the labelled induced subgraph ordering.
\end{theorem}

\begin{proof}
First, by the comments at the start of this section, $\{S_k^\circ\}$ is an infinite antichain in the class of bichain graphs, under the labelled induced subgraph ordering.

Next, consider any subclass $\C$ of the bichain graphs which has only finite intersection with the antichain $\{S_k^\circ\}$. This means that there exists $k$ such that $\C\subseteq\Free(S_k,S_{k+1},\dots)$. 

By Lemma~\ref{lem:induced-sk}, for each canonically prime graph $G$ in $\C$ there exists $n$ such that $G$ can be embedded in $Z_{n,2k}$. This implies, by Lemma~\ref{lem:letter11}, that all canonically prime graphs in $\C$ are $2k$-letter graphs. 

By Theorem~\ref{thm:letter11}, the canonically prime graphs of $\C$ are thus well-quasi-ordered with respect to the labelled induced subgraph relation, and Theorem~\ref{thm:canonical} allows us to conclude that the whole subclass $\C$ is well-quasi-ordered with respect to the labelled induced subgraph relation, completing the proof.
\end{proof}

The proof that split permutation graphs contains a canonical labelled antichain is analogous. The antichain in this case is built upon a sequence of graphs $T_k$ for which $T_k^*=S_k$ (i.e.\ each graph $T_k$ is formed from $S_k$ by replacing one part of the bipartition with a clique). The labelled infinite antichain is then $\{T_k^\circ\}$, in which the first and last pair of vertices of each graph are coloured differently from the others (here, copies of $P_4$ must embed in consecutive pairs of vertices, rather than the copies of $2K_2$ found in $S_k$). 

To see that the grid $Z_{n,k}^*$ (formed by connecting together all vertices in all even columns) is a $k$-letter graph, one needs only add pairs of letters corresponding to even columns in the set of connections given in the proof of Lemma~\ref{lem:letter11}. Next, although we cannot directly use the notion of canonically prime, we can define an analogue for split graphs with the same three constructions of disjoint union, join and skew join, but where one side of the partition is always a clique.  With this, the equivalent to Theorem~\ref{thm:canonical} is readily obtained, following the techniques of~\cite{atminas:labelled-induced:}. The rest of the argument to prove that $\{T_k^\circ\}$ is a canonical labelled antichain follows.

Finally, it is worth observing that the antichain $\{T_k^\circ\}$ of split permutation graphs is well-known in the study of permutation classes, since it corresponds to the ``Widdershins antichain'' of permutations (because of the way it appears to spirals anticlockwise). Thus, another (largely analogous) method to establish that $\{T_k^\circ\}$ is canonical for split permutation graphs would be to build on the structural results of Albert and Vatter~\cite{albert:generating-and-:}.

\section{Concluding remarks}
\label{sec:con}

In this paper, we have exhibited two classes that possess the simultaneous properties of being finitely-defined, minimal of unbounded clique-width, and of containing a canonical labelled infinite antichain. We observe that while the class of split permutation graphs is self-complementary in the sense that $G$ belongs to the class if and only if the complement of $G$ does, this is not the case for bichain graphs. We can, therefore, describe one more class possessing all three properties (finitely-defined, minimal of unbounded clique-width, and of containing a canonical labelled infinite antichain), namely, the class of complements of bichain graphs. We believe that various other graph transformations  (e.g.\ other sequences of local complementations) may lead to more examples of such classes.

The existence of classes possessing the first two of these three properties are sufficient to disprove Conjecture~8 of Daligault, Rao and Thomass\'e~\cite{daligault:well-quasi-orde:}. However, in the same paper, the authors proposed another conjecture that relate the second two properties, namely labelled well-quasi-ordering and clique-width. Restricting our ordered set of labellings $(W,\leq)$ to the special case where $|W|=2$ and $W$ is an antichain, the authors of~\cite{daligault:well-quasi-orde:} conjecture that a hereditary graph class which is well-quasi-ordered under labellings by $W$ must have clique-width that is bounded by a constant. 

Note that this is equivalent to asking whether every \emph{minimal} hereditary class of unbounded clique-width has a labelled infinite antichain that uses at most two labels. Indeed, if a hereditary class $\C$ of unbounded clique-width does not contain a minimal hereditary class of unbounded clique-width, then one can construct an infinite strictly decreasing sequence of graph classes $\C=\C_0\supset \C_1\supset \C_2\ldots$ of unbounded clique-width such that $\C_{i+1}$ is obtained from $\C_i$ by forbidding a graph $G_i\in \C_i$ as an induced subgraph. In this case, the sequence $G_0,G_1,G_2,\ldots$ creates an (unlabelled) infinite antichain in $\C$. Therefore, the conjecture of Daligault, Rao and Thomass\'e can be restricted, without loss of generality, to minimal hereditary classes of unbounded clique-width. 

Prior to this paper, only two minimal classes with unbounded clique-width were known: bipartite permutation graphs and unit interval graphs~\cite{lozin:minimal-classes:}. In both cases, they also contain a canonical (unlabelled) infinite antichain. When combined with analogous results for the two classes we have considered in this paper, we propose the following stronger conjecture.

\begin{conj}\label{conj-canonical}
Every minimal hereditary class of graphs of unbounded clique-width contains an infinite set of graphs that forms a canonical labelled infinite antichain that uses at most two incomparable labels.
\end{conj}

Note that dropping the requirement for \emph{labelled} well-quasi-ordering is already known to be false: it was recently shown in~\cite{lozin:wqo-vs-cw} that there exists a class of graphs that has unbounded clique-width, but which is well-quasi-ordered with respect to the induced subgraph relation. This class, however, contains an infinite antichain that uses two labels.

\bibliographystyle{acm}
\bibliography{../refs}

\end{document}